\documentclass[reqno,11pt]{amsart} 
\usepackage{esint}
\usepackage{url}    
\usepackage{mathtools}    
\RequirePackage[colorlinks,citecolor=blue,urlcolor=blue]{hyperref} 

\hoffset= - 0.5 cm \voffset= - 0.1 cm \usepackage{color}       
\newtheorem{theorem}{Theorem}[section]
\newtheorem{lemma}[theorem]{Lemma} 
\newtheorem{proposition}[theorem]{Proposition}
\newtheorem{corollary}[theorem]{Corollary}
\newtheorem{remark}[theorem]{Remark}
\newtheorem{definition}[theorem]{Definition}
\newtheorem{example}{Example}

\def\qed{\hfill \hbox{\hskip 6pt\vrule
width6pt height6pt depth1pt  \hskip1pt}
\smallskip}

\newcommand{\bth}{\begin{theorem}}
\newcommand{\eth}{\end{theorem}}
\newcommand{\bpr}{\begin{proposition}}
\newcommand{\epr}{\end{proposition}} 
\newcommand{\bco}{\begin{corollary}}
\newcommand{\eco}{\end{corollary}}
\newcommand{\ble}{\begin{lemma}}
\newcommand{\ele}{\end{lemma}}
\newcommand{\bpf}{\begin{proof}}
\newcommand{\epf}{\end{proof}}
\newcommand{\bex}{\begin{example}}
\newcommand{\eex}{\end{example}}
\newcommand{\bdf}{\begin{definition}}
\newcommand{\edf}{\end{definition}}
\newcommand{\bre}{\begin{remark}}
\newcommand{\ere}{\end{remark}}

\newcommand{\beq}{\begin{equation}}
\newcommand{\eeq}{\end{equation}}
\newcommand{\bal}{\begin{aligned}}
\newcommand{\eal}{\end{aligned}}
\newcommand{\ben}{\begin{enumerate}}
\newcommand{\een}{\end{enumerate}}
\newcommand{\beqr}{\begin{eqnarray*}}
\newcommand{\eeqr}{\end{eqnarray*}}

\def\P{{\mathbb P}}
\def\R{{\mathbb R}}

\def\E{{\mathbb E }}

\def\lan{{\langle}}
\def\ran{{\rangle}}
\def\hh{{\vskip 2mm \noindent }}
\def\vv{{\vskip 1mm}}

\begin{document}

\title[ One-side Liouville theorems ]
{ One-side  Liouville  theorems under    an  exponential  growth condition for  Kolmogorov operators
}

\author{E. Priola}
\thanks{  The author  is a  member of GNAMPA  of the Istituto Nazionale di Alta Matematica (INdAM)}

 \address{ Enrico Priola  
   Dipartimento di  
 Matematica,  
Universit\`a  di Pavia, Pavia, Italy  }  
\email{enrico.priola@unipv.it}


    
\begin{abstract} It is known  that   
 for a possibly degenerate hypoelliptic  Ornstein-Uhlenbeck operator   
$$
L= \frac{1}{2}\text{ tr}  (QD^2 ) + \langle Ax, D \rangle =  \frac{1}{2}\text{ div}  (Q  D ) + \langle Ax, D \rangle,\;\; x \in \R^N,
$$
all (globally) bounded solutions of $Lu=0$ on $\R^N$ are constant if and only if all the eigenvalues of $A$ have non-positive  real parts
(i.e., $s(A) \le 0)$.
 We show that if $Q$ is positive definite and $s(A) \le 0$, then  any
non-negative solution $v$ of $Lv=0$ on $\R^N$ which   has at most an exponential   growth is indeed constant. 
 Thus under a non-degeneracy condition  we relax  the  boundedness assumption on the harmonic functions and maintain the sharp condition on the eigenvalues of $A$. 
 We also prove   a related one-side Liouville theorem 
in the case of   hypoelliptic Ornstein-Uhlenbeck operators. 
 \end{abstract}
 
\maketitle

\section{ Introduction} 

 Let $Q$ be a symmetric  non-negative definite $N \times N$ matrix and let $A$ be a real 
$N \times N$ matrix. The possibly degenerate   Ornstein-Uhlenbeck operator (briefly OU operator)  
 associated with  $(Q,A)$ is defined as    
\begin{equation}\label{ss3} 
L = \frac{1}{2}\text{ tr}  (QD^2 ) + \langle Ax, D \rangle =  \frac{1}{2}\text{ div}  (Q  D ) + \langle Ax, D \rangle, \;\; x \in \R^N.
\end{equation} 
We will assume the so-called  Kalman  controllability  condition (see, for instance,  Chapter 1 in \cite{zabczyk}): 
 \begin{equation} \label{kal} 
 \text{rank}[\sqrt{Q}, A \sqrt{Q}, \ldots, A^{N-1} \sqrt{Q} ] =N.
\end{equation}
This is equivalent to the  hypoellipticity of $L- \partial_t$ (see  \cite{kup_1972} and \cite{lanconelli_polidoro_1994} for more details). Clearly, in the non-degenerate case when  $Q$ is positive definite   we have that condition 
\eqref{kal} holds.     
 
 We study one-side  Liouville type  theorems for $L$, i.e., we want to  know when  non-negative smooth 
 solutions $u : \R^N \to \R$ to
 $$
 Lu(x) =0,\;\; x \in \R^N,
 $$
  are constant. Such functions  are called positive (non-negative) harmonic functions for $L$  (see  \cite{Pinskii}).  
  Positive harmonic functions are important in the study of the Martin boundary for $L$ (see  \cite{COR} for the non-degenerate two-dimensional case  and Chapter 7 in   \cite{Pinskii} for more information on the Martin boundary for non-degenerate diffusions).
   
  Before stating our main result we  discuss 
  a known   
  Liouville theorem concerning  bounded harmonic functions (briefly BHFs) for $L$.  To this purpose  we introduce  
  the spectral bound of $A$: 
 \begin{equation} \label{saa}
  s(A) = \max    \{ Re (\lambda)\; :\; \lambda \in \sigma (A) \}.
\end{equation} 
 where $\sigma (A)$ is the spectrum of $A$ (i.e., set of the  eigenvalues of $A$).  
 By Theorem 3.8 of \cite{PZjfa} it follows that   all BHFs for the  hypoelliptic OU operator $L$  are constant if and only if $s(A) \le 0$ 
  (cf.   Theorem \ref{jfa} and see the comments in Section 2 for more details). 
   Liouville theorems involving  BHFs  for diffusions   are
   considered, for instance,  in  \cite{Pinskii},
  \cite{BF}, 
  \cite{PZjfa} and \cite{PW}. Liouville theorems involving  BHFs for purely  non-local OU operators are proved in \cite{Wang11} and \cite{Schilling}; see also Section  \ref{open}.    
      Recall that  Liouville theorems for BHFs  have  probabilistic interpretations, in terms of absorption functions (see Chapter 9 in \cite{Pinskii} and Section 6 in \cite{PZjfa}) and in terms of  successful couplings (see \cite{Wang11}, \cite{Schilling} and the references therein).

Theorem 3.8 in  \cite{PZjfa} suggests the natural question if under the assumption $s(A) \le 0$ we have more generally  a one-side Liouville type theorem for $L$.  In general, this is an open problem (see also Section \ref{open}).  
  However  one-side Liouville theorems for $L$ 
have been proved in some  cases under specific assumptions on $Q$ and $A$ (such assumptions  imply that  $s(A) \le 0$). We refer to \cite{KL0}, \cite{KL1}, \cite{KL2}, \cite{KLP}  and the references therein; 
 the papers 
\cite{KL0}, \cite{KL1} and  \cite{KL2} contain also one-side Liouville theorems for other classes of Kolmogorov operators.  
  We mention  
  the main result in  \cite{KLP} where it is proved a one-side Liouville theorem assuming that $Q$ is positive definite and that the norm of the exponential matrix $e^{tA}$ is uniformly bounded when $t \in \R$ (cf. Theorem \ref{ou1} and the related discussion in  Section 2).

The main result of this work states that if $Q$ is positive definite   and $s(A) \le 0$, then  any 
 positive harmonic function $v$ for   $L$  which   has at most an exponential   growth is indeed constant. 
 Thus under a non-degeneracy condition  we relax  the  boundedness assumption on the harmonic functions of       \cite{PZjfa}   and maintain the sharp condition on the eigenvalues of $A$ (see Theorem \ref{main} for the precise statement). 
 We also prove   a related one-side Liouville theorem under  a sublinear growth condition which is valid 
 for   hypoelliptic OU operators (see Theorem \ref{n}).

 \vv The plan of the paper is the following one.
 
We recall and discuss  known results in Section 2. In Section 3 we prove the convexity of positive harmonic functions for $L$ under an exponential growth condition. This is a consequence of  a result given in \cite{PZproc}. Such convexity property will be the starting point for  the proof of our main result.
  We state and discuss  Theorem \ref{main}   in Section 4 where we also present  Example \ref{exam} to illustrate an idea of the proof in a significant case. The complete proof of the main result is given in Section 5.
We finish the paper by presenting some open problems.

 \subsection{Notations}   We denote by $|\cdot | $ the usual euclidean norm in any $\R^k$, 
 $k \ge 1$. Moreover,
 $ x \cdot y$ or $\langle x, y  \rangle$ indicates the usual inner product in $\R^k$, $x,y \in \R^k.$ The canonical basis in $\R^k$ is denoted by $(e_i)_{i=1,...,k}.$

Let $k \ge 1$. Given a regular function  $u: \R^{k}$, we denote by 
 $D^2 u (x)$ the $k \times k$ Hessian matrix of $u$ at $x \in \R^{k}$, 
  i.e.,  $D^2 u (x) = (\partial_{x_i x_j
}^2 u   (x))_{i,j =1, \ldots, k}$,  where $\partial_{x_i x_j
}^2 u   $ are the usual second order  partial derivatives 
of $u$.   Similarly we define the gradient $D u(x) \in \R^k$. 
 
 Given a real $k \times k $ matrix $A$, $\| A \|$ denotes its operator norm
and $\text{tr} (A)$ its trace.

Given a symmetric non-negative definite  $k \times k$ matrix  $Q$ we denote by $N(0,Q)$ the symmetric Gaussian
measure with mean 0 and covariance matrix $Q$  (see, for instance, Section 1.7 in \cite{Baldi} or Section 2.2 in \cite{DP14}). If in addition $Q $ is positive definite than $N(0,Q)$ has the  density $\frac{1}{ \sqrt {(2 \pi)^k \, \text{det}(Q)}}e^{- \frac{1}{2}\langle   
Q^{-1} x, x\rangle},$  $x \in \R^k,$ with
respect to the $k$-dimensional Lebesgue measure.  

\section{Some known  results  }     
 
   First we introduce  the Banach space $B_b(\R^N)$
 of all Borel and bounded functions from $\R^N$ into  $\R$ endowed with the supremum  norm.  We define the following semigroup of operators $(P_t)$ acting on $B_b(\R^N)$:
\begin{gather} \label{oo1}
(P_t f)(x)=   P_t f (x) = \int_{\R^N }  f (e^{t A}x \, + y )     N(0 ,Q_t)dy, \;\;
x \in \R^N, \, t > 0,
\end{gather}
 $P_0 f = f$,  $f \in B_b(\R^N)$;  here $N(0,Q_t)$ is the Gaussian measure with mean 0 and covariance matrix
\begin{gather*} 
Q_t = \int_0^t e^{sA} Q e^{sA^*} ds 
\end{gather*}
(see, for instance,  Chapter 6 in \cite{DP02} for more details).
 We are using  exponential matrices $e^{sA}$ and $e^{sA^*}$ where $A^*$ denotes the transpose of the matrix $A$;
 $(P_t)$   is called the  Ornstein-Uhlenbeck  semigroup (brefly the OU semigroup).  
 
 Recall that the Kalman condition \eqref{kal} is equivalent to the fact that $Q_t$ is positive definite  for $t>0$ (cf. Section 1.3 in \cite{zabczyk}). It is also equivalent to the strong Feller property of  $(P_t)$ and, moreover, to the fact that  $P_t (B_b(\R^N)) \subset C^{\infty}_b(\R^N)$, $t>0$ (see, for instance,   Chapter 6 in \cite{DP02}).  
 

\vv Now we mention  a special case 
  of Theorem 3.8 in  \cite{PZjfa} (this  covers also some classes of  non-local OU operators). 

\begin{theorem} [\cite{PZjfa}] \label{jfa} Let us consider the  hypoelliptic OU operator $L$ (i.e., we assume \eqref{kal}$).$ Let $w \in C^2(\R^N)$ be a bounded solution to $Lw (x) =0$ on $\R^N$ $($\footnote{$L$ is hypoelliptic, so that every distributional solution to $L w =0$ is of class $C^\infty$. 
}$).$ 
   Then  $w$ is constant if and only if $s(A) \le 0$.
\end{theorem}  

\begin{remark} {\em 
Note that a smooth bounded real function $w$ is  a  solution to $Lw (x) =0$ on $\R^N$ if and only if  it is a bounded harmonic functions  for the OU semigroup, i.e., 
$$   
P_t w =w\;\; \text{on $\R^N$,}\;\;   t \ge 0.  
$$
This fact can be easily proved  by using the It\^o formula (we point out  that the It\^o formula    will be also  used in the proof of Proposition \ref{first}). }            
\end{remark} 
   
To prove $\Longleftarrow$ in the previous theorem one uses the next result, the proof of which uses control theoretic techniques.

\begin{theorem} [\cite{PZnull}] \label{nullo} Let us consider the matrix $Q_t^{-1/2}e^{tA}$, $t>0$ (this is well-defined by \eqref{kal}$).$ We have
\begin{equation}\label{nu}
 \|Q_t^{-1/2}e^{tA} \| \to 0 \;\; \text{as} \; t \to \infty  \;\; \Longleftrightarrow \; \; s(A) \le 0.
\end{equation} 
 \end{theorem} 

Now we  state  a one-side Liouville type theorem proved in  Theorem 1.1 of \cite{KLP} which will be used in the sequel.  

\begin{theorem}[\cite{KLP}]   \label{ou1} Let $Q$ be a $N \times N$ positive definite matrix.  
  Suppose that $\sup_{t \in \R} \| e^{tA}\| < \infty$  $($\footnote{This is equivalent to require that 
$A$ is diagonalizable over $\mathbb{C}$
with all the eigenvalues on the imaginary axis;  in particular this implies that $s(A) \le 0.$}$)$.  
  Let $v$ be a smooth non-negative solution to $L v=0$ on $\R^N.$
 Then  $v$ is constant. 
 \end{theorem} 
   Clearly, we can replace in the previous theorem the condition that $v$ is non-negative by requiring  that $v$ is bounded from above or from below. 
\begin{remark} 
 \label{qq} {\em 
(i) Theorem 1.1 in \cite{KLP}  is proved when $Q=I$. On the other hand,
 if $Q$ is positive definite then by using the change of variable   $u(x) = 
 l(Q^{-1/2} x)$, $l(y) = u(Q^{1/2}y)$, $y \in \R^N,$ we can pass from an OU operator associated with  $(Q,A)$ to an OU operator 
 associated with  $(I, Q^{-1/2}A Q^{1/2})$. 
   In particular,  we have    $Lu(x) =0$, $x \in \R^N$, if and only if 
 $$ 
 \frac{1}{2}\triangle l(y) + \langle   Q^{-1/2}A Q^{1/2} y, D l(y)\rangle =  0,\;\;\; y \in \R^N.
 $$ 
 Therefore,  Theorem 1.1 in \cite{KLP}  holds more generally when 
 $Q$ is positive definite. 

 (ii) We do not know if the previous theorem  holds replacing the assumption  that $Q$ is positive definite with  the more general   Kalman condition \eqref{kal}.
 }
 \end{remark}

\section{Positive harmonic function for the OU semigroup}

Note that formula  \eqref{oo1} is meaningful even if the Borel function $f$ is only non-negative.
 Following \cite{PZproc}
 we say that a Borel function $u: \R^N \to \R_+$
  is a {\sl positive harmonic function for the OU semigroup $(P_t)$}
if it satisfies 
  \begin{equation}\label{harm}  
 P_t u(x) = u (x),\;\;  \; x \in \R^N,\;\; t \ge 0.
\end{equation} 
   
The next theorem is a special case of  Theorem 5.1 in \cite{PZproc} which holds in infinite dimensions as well. Its proof uses Theorem \ref{nullo} together with an idea of  S. Kwapien (personal communication). We sketch  its proof in Appendix.
    

\begin{theorem} \label{ci}
Assume the Kalman condition \eqref{kal} and $s(A) \le 0$. Consider a  positive  harmonic function $u$ for the OU semigroup $(P_t)$.
 Then $u$ is convex on $\R^N$.    
\end{theorem}


In the next result we provide a sufficient condition under which 
 positive harmonic functions for $L$ are positive  harmonic functions for the OU semigroup as well. 
For such result we do not need the Kalman condition \eqref{kal}.
 
\begin{proposition}\label{first} Let $u \in C^2(\R^N)$ be a non-negative  solution to $Lu (x) =0$ on $\R^N$.
 Assume that $u$ verifies the following exponential growth condition:  
 there exists $c_0 >0$  such that  
\begin{equation}\label{e1} 
 |u(x)| \le  c_0 \, e^{ c_{0}\, |x|} ,\;\;\; x \in \R^N.
\end{equation}  
Then we have $P_t u (x) = u(x)$, $x \in \R^N$, $t \ge 0$ $($\footnote{ 
 This result holds more  generally assuming that there exists $\epsilon \in [0,2)$ and  $C_{\epsilon} >0$ such that $|u(x)| \le C_{\epsilon} e^{C_{\epsilon} \, |x|^{2-\epsilon}},$ $x \in \R^N$.    Moreover, the hypothesis  that $u$ is non-negative is not necessary.
  }$).$   
\end{proposition} 
\begin{proof} The proof uses stochastic calculus.
Let us introduce the OU stochastic  process starting at $x \in \R^N$ (see, for instance, page 232 in \cite{IW}).  It is the solution to  the following  SDE  
 \begin{equation}\label{ou34}
 X_t^x =  x + \int_0^t A X_s^x ds
   \, + \,
 \int_0^t \sqrt{Q}  \,  d W_s,\, \;\; t \ge 0,\;\; x \in \R^N.
 \end{equation} 
Here $W = (W_t)$ is a standard $N$-dimensional Wiener process defined and adapted on  a stochastic basis $( \Omega,  {\mathcal  F}, ({\mathcal  F}_t),  \P)$. The solution is given by  
\begin{gather*} 
  X_t^x = e^{tA} x + \int_0^t e^{(t-s)A} \sqrt{Q}dW_s.
\end{gather*}
 It is well-known that 
 $\E [ u(X^x_{t })] = P_t u(x),$ $t \ge 0,$ $x \in \R^N$   (see, for instance, Section 5.1.2 in \cite{DP14}).  
 
 \vskip 1mm
 By It\^o's formula (see, for instance, Chapter 8 in \cite{Baldi} or   Chapter 2 in \cite{IW}) we know that, $\P$-a.s.,    
\begin{gather*}
u(X^x_t) = u(x) + \int_0^t Lu (X_s^x) ds + M_t = u(x) + M_t,  \;\; t \ge 0,\, x \in \R^N,
\end{gather*}
 using   the local martigale $M= (M_t)$,  $M_t = \int_0^t  D u(X_s^x)  \sqrt{Q} dW_s$. Let us fix $x \in \R^N.$
  By using the stopping times $\tau_n^x = \inf \{ t \ge 0 \, :\, X_t^x \in B_n \}$ (here $B_n$ is the open ball of radius $n$ and center 0) we find  
 \begin{gather*}
\E [ u(X^x_{t \wedge \tau_n^x})] = u(x) + \E [ M_{t \wedge \tau_n^x}].
\end{gather*}
 By considering 
a $C^2_b$-function   $u_n$ with bounded first and second derivatives on $\R^N$ which coincides with $u$ on $B_{n+1}$ we obtain 
\begin{gather*}
 M_{t \wedge \tau_n^x} =   \int_0^{t \wedge \tau_n^x}   D u_n(X_s^x) \,\sqrt{Q} dW_s, \;\; \;\;\; t \ge 0,\;\; n \ge 1. 
\end{gather*}
 Since  $\int_0^{t }   D u_n(X_s^x)  \sqrt{Q} dW_s$ is a martingale, by the Doob optional stopping theorem we know  that 
 \begin{gather*}
0=\E \Big[ \int_0^{t \wedge \tau_n^x}   D u_n(X_s^x) \, \sqrt{Q} dW_s \Big] =  \E [ M_{t \wedge \tau_n^x}], \;\;\; t \ge 0,\,  n \ge 1.  
\end{gather*}
 We arrive at
 \begin{gather} \label{e2}
\E [ u(X^x_{t \wedge \tau_n^x})] = u(x),\;\; t \ge 0,  n \ge 1. 
\end{gather}
We fix $t>0.$ In order to pass to the limit in \eqref{e2} we use that, for any $n \ge 1$, $\P$-a.s.,  
\begin{gather*}
| u(X^x_{t \wedge \tau_n^x})| \le c_{0} \, e^{c_{0} \,  |(X^x_{t \wedge \tau_n^x}|^{} }  
\le  c_{0}  \, e^{c_{0}  \, \sup_{s \in [0,t]} |X^x_{s}|^{} }.
\end{gather*}    
It is known that 
there exists $\delta >0$, possibly depending also on $t$,  such that 
 \begin{equation}\label{fern}
\E \Big[ \exp \Big (\delta \sup_{s \in [0,t]}  
\Big | \int_0^s e^{(s-r)A} \sqrt{Q} \, dW_r \Big|^2   \Big) \Big] < \infty
\end{equation} 
(for instance, one can use Proposition 8.7 in \cite{Baldi} together with the estimate  $ \sup_{s \in [0,t]} \big| \int_0^s e^{sA}  e^{-rA} \sqrt{Q} \, dW_r \big| $ $\le c_t  \sup_{s \in [0,t]} \big| \int_0^s  e^{-rA} \sqrt{Q} \, dW_r \big| $).  
 It follows that $\E [ e^{c_{0}  \, \sup_{s \in [0,t]} |X^x_{s}|^{ } }] < \infty$. 
 Since  $\P$-a.s. $\tau_n^x \to \infty$,  
 we can pass to the limit in \eqref{e2} as $n \to \infty$ by the dominated convergence theorem and get 
 \begin{gather*}
\E [ u(X^x_{t })] = u(x),\;\; t \ge 0.
\end{gather*}
 The proof is complete since $\E [ u(X^x_{t })] = P_t u(x)$.
   \end{proof} 

According to Theorem \ref{ci} and Proposition   \ref{first} we have  
\begin{corollary} \label{ci2}
Assume the Kalman condition \eqref{kal} and $s(A) \le 0$. Let $u$ 
 be a non-negative smooth solution to $Lu (x) =0$ on $\R^N$ which verifies
the  exponential growth condition \eqref{e1}.     
   
   Then $u$ is a convex function on $\R^N$.
\end{corollary} 
 
  
We obtain here a first  one-side Liouville type theorem for possibly degenerate hypoelliptic OU operators. It holds under a sublinear growth condition and generalizes Theorem \ref{jfa}.
     
\begin{theorem} \label{main0} Let us consider the 
OU operator $L$ under the Kalman condition \eqref{kal}.  Let $u \in C^2(\R^N)$ be a non-negative solution to $Lu  =0$ on $\R^N$. Assume  
 the following growth condition: there exist $\delta \in [0,1)$ and $C_{\delta}>0 $   
such that  
\begin{equation}\label{sublin}   
|u(x)| \le C_{\delta}\, (1+ |x|^{\delta}),\;\; x \in \R^N.
\end{equation}  
If  $s(A) \le 0$ then we have that $u$ is constant.
\end{theorem}
\begin{proof}  Since condition \eqref{sublin} implies the exponential growth condition \eqref{e1} we can apply   Corollary \ref{ci2}.   The assertion follows since $u $ is a convex function on $\R^N$.
 \end{proof}

 \section{ A one-side  Liouville theorem under an exponential   growth condition}
 
 The main result of the paper concerns  non-degenerate 
OU operators $L$:
 \begin{equation}\label{ss0} 
L = \frac{1}{2} \text{ tr}  (QD^2 )   + \langle Ax, D \rangle, 
\end{equation}
where $Q$ is a  positive definite  $N \times N$-matrix.

\begin{theorem} \label{main} Let us consider the  OU operator $L$ with $Q$ positive definite. Let $u \in C^2(\R^N)$ be a non-negative solution to $Lu (x) =0$ on $\R^N$.  Suppose that  $s(A) \le 0$ holds. Suppose that $u$ satisfies  the exponential 
growth condition \eqref{e1}. 
 Then   $u$ is constant.
\end{theorem}

In order to prove the result we need a first lemma which holds more generally for hypoelliptic OU operators.

\begin{lemma} \label{n} Let us consider 
the hypoelliptic OU operator $L$. Let $u \in C^2(\R^N)$ be a non-negative solution to $Lu (x) =0$ on $\R^N$. Suppose that $u$ verifies  the exponential growth condition 
 \eqref{e1}.  
Then if $s(A) \le 0$ we have  the following identity for any $x_0, \, x \in \R^N $  
 \begin{gather} \label{serve}
u(x) \ge u(x_0) -  Du(x_0)\cdot x_0  +   Du(x_0)\cdot e^{tA}x,\;\; 
\;\; t \ge 0.
\end{gather}
\end{lemma}
\begin{proof}  We know by Corollary \ref{ci2} that 
 for any $x_0 \in \R^N $ we have
 \begin{gather} \label{si}
u(y) \ge u(x_0) + Du(x_0)\cdot (y-x_0),\;\;\; y \in \R^N.
\end{gather} 
 We apply the OU semigroup $(P_t)$  to both sides of \eqref{si}:
\begin{gather*}
P_t u (x) \ge u(x_0) + Du(x_0) \cdot \int_{\R^N }  [e^{t A}x - x_0  \, + z ]     N(0 ,Q_t)dz
\\
= u(x_0) -  Du(x_0)\cdot x_0  +   Du(x_0)\cdot e^{tA}x.
\end{gather*}   
Taking into account Proposition \ref{first} we have that 
$P_t u = u $, $t \ge 0$, and  the assertion follows.
  \end{proof}

 To illustrate  the  proof of Theorem \ref{main} we first examine
 an
 example
 when  $N=3$.
 
\begin{example} \label{exam}{\em  We introduce  
\begin{gather*}   
 A  = \left(  
           \begin{array}{ccc}
             0  &  1  & 0 \\
              0  & 0 & 1 \\
               0  & 0 & 0
            \end{array} 
          \right), \;\;\;\;  e^{tA} x= \left(
           \begin{array}{c}
             x_1 +  t x_2  +    \frac{t^2}{2} x_3   \\
             x_2 +   t x_3    \\
                x_3  
            \end{array} 
          \right),
\end{gather*} 
 $x = (x_1, x_2, x_3) \in \R^3$.   Let $Q$ can  be any positive definite $3 \times 3 $ matrix. 
  We consider the  OU operator associated to $(Q,A)$.

  We first prove that, for any $x_0 \in \R^3$,  $\partial_{x_1}u(x_0)=\partial_{x_2}u(x_0) =0.$  
  Suppose by contradiction that  
 \begin{gather*}  
k = \partial_{x_1}u(x_0) \not =0,
\end{gather*}
 for some $x_0 \in \R^3$.
 Then we consider  $x= (0, 0, k)$ and  by \eqref{serve} we get
 \begin{gather*}
 u(0, 0, k) \ge  u(x_0) -  Du(x_0)\cdot x_0  +   Du(x_0)\cdot e^{tA} (0, 0, k)
 \\
 = u(x_0) -  Du(x_0)\cdot x_0  + Du(x_0)\cdot   \big(\frac{t^2}{2} k,   t k, k  \big) 
 \\
 =  u(x_0) -  Du(x_0)\cdot x_0  +  \frac{t^2}{2} \,  k^2
  + t \partial_{x_2}u(x_0) k + \partial_{x_3}u(x_0)k,\;\;\;\; t \ge 0. 
\end{gather*}
 Letting $t \to \infty$ we get a contradiction since $\frac{t^2}{2} \,  k^2
  + t \partial_{x_2}u(x_0) k $ tends to $\infty$.   It follows that 
 $\partial_{x_1}u(x_0)  =0$. We have 
 $
u(x_1, x_2, x_3) = u(0,x_2, x_3)$   on $ \R^3. $
 
 \vskip 1mm
Suppose by contradiction that $
l = \partial_{x_2}u(x_0) \not =0,
$ for some $x_0 \in \R^3$.  
 Then we consider  $x= (0, 0, l)$ and  by \eqref{serve} for any $t \ge 0$ we get
 \begin{gather*}
 u(0, 0, l) \ge  
  u(x_0) -  Du(x_0)\cdot x_0  + Du(x_0)\cdot   (\frac{t^2}{2} l,   t l, l  )
 \\
 =  u(x_0) -  Du(x_0)\cdot x_0       
  + t  l^2 + \partial_{x_3}u(x_0)l.   
\end{gather*}
 Letting $t \to \infty$ we get a contradiction. It follows that 
 $\partial_{x_2}u(x_0)  =0$, for any $x_0 \in \R^N$. We have obtained  that 
 $ u(x_1, x_2, x_3) = u(0,0, x_3) = v(x_3)$   on $ \R^3. $ 

 \vskip 1mm
 Since   $0= Lu(x) = \frac{q_{33}}{2} \partial_{x_3 x_3}^2 u(0,0,x_3) =  \frac{q_{33}}{2}  \frac{d^2}{dx_3^2}     v(x_3)$, 
  $x_3 \in \R,$ where $q_{33} = Q e_3 \cdot e_3$,
   we finally obtain    that $u=cost$. } 
 \end{example}    

 In the proof of Theorem \ref{main} we will also
use
the following     remarks. 
  \begin{remark} \label{qq1}  {\em Arguing as in (ii) of Remark \ref{qq} we note that
 it is enough to prove Theorem \ref{main} when $Q$ is replaced by
 \begin{gather*}
\delta\, Q 
\end{gather*} 
for some $\delta >0$. Indeed 
 by using the change of variable   $u(x) = 
 v(\delta^{1/2} x)$,   $v(y) = u(\delta^{-1/2} y)$,    we can pass from an OU operator associated with  $(Q,A)$ to an   OU operator
 associated with   $(\delta Q, A )$. 
 We have 
 \begin{gather*}
L u(x)  =0,\;\; x \in \R^N    \; \Longleftrightarrow   
\;      \frac{\delta}{2}   tr (QD^2 v(y))   + \langle Ay, D v(y)\rangle =0,\; y \in \R^N.    
\end{gather*}
  } 
 \end{remark}
 
\begin{remark}\label{prova}{\em
In the sequel we will  always assume that in \eqref{ss0}  
\begin{equation}\label{assume}
\text{\it the matrix  $A$ is in the real Jordan form,}   
\end{equation}
{ possibly replacing $Q$ by $PQ P^{*}$ where $P$ is a $N \times N $ real invertible matrix. }
\vv Note that $PQ P^{*}$ is still positive definite. Let us clarify the previous  assertion.
Let  $P$ be an invertible real matrix such that $PAP^{-1} = J  $.
 
  Using the change of variable   $u(x) = 
 v(P x)$,   $v(y) = u(P^{-1} y)$,    we can pass from an OU operator associated with  $(Q,A)$ to an   OU operator
 associated with   $(P QP^*, PAP^{-1}  )$. 
 We have 
 \begin{gather*}
L u(x)  =0,\;\; x \in \R^N    \; \Longleftrightarrow   
\;      \frac{1}{2}   tr (PQP^* D^2 v(y))   + \langle J y, D v(y)\rangle =0,\; y \in \R^N.    
\end{gather*}
We also remark  that $s(J) = s(A).$
 }\end{remark}

\section{On the proof of Theorem \ref{main}}

According to Remarks \ref{qq1}  and \ref{prova}  we  concentrate on proving the  Liouville theorem for $L$ in \eqref{ss0} assuming that {\sl  $A$ is in the real Jordan form}. 
 Moreover,  when  it will be  needed { we will replace  
$Q$   by $\delta Q$ with $\delta >0$ small enough.}    
 

\subsection{A technical  lemma}

 Recall that $(e_j)$ denotes the canonical basis in $\R^N$. Given a real $N \times N $ matrix $C$ we write  $C    =   B_0 \oplus B_1 \oplus ... \oplus B_n$ if 
\begin{gather*} \label{somma}
C = \begin{pmatrix} 
 B_0 & 0 & ... & 0\\
 0 & B_1 & ... & 0 \\
  0 & ... & ... & 0
 \\
 0 & ... & 0 & B_n
\end{pmatrix},
\end{gather*}
where  $B_i$ is a   real $k_i \times k_i$ matrix, $i = 0, ..., n$, $k_i \ge 1$ and $k_0 + ... + k_n = N.$ 
 Let $x \in \R^N$. We write
\begin{gather} \label{a34}
x= (x_{B_0},..., x_{B_n})
\end{gather}
where  $x_{B_0}= (x_1^0, ..., x_{k_0}^0)$ $= (x_1, ..., x_{k_0})$, 
 \begin{gather} \label{17bis}
 x_{B_i}= (x_1^i, ..., x_{k_i }^i )= (x_{k_0 + ... + k_{i-1}+1}, ..., x_{k_{i-1}+ k_i}).
\end{gather}  
 We say that $x_1^i, ..., x_{k_i }^i$ are the variables corresponding  to $B_i.$  
  
  We also consider the related  orthogonal projections $\pi_{B_i}: \R^N \to \R^{k_i}$:
 \begin{equation}\label{orto}  
 \pi_{B_i}x = x_{B_i}
 \;\;\;\; x \in \R^N, \, i =0, ..., n.
\end{equation}
Clearly, we have   
\begin{equation} \label{stt}
 Cx = \big( B_0 x_{B_0}, ...,  B_n x_{B_n}\big) = \big( B_0 \pi_{B_0}x, ...,  B_n \pi_{B_n}x\big),\;\; x \in \R^N.
\end{equation}

Now we    write    the  $N \times N$ matrix  
$A$ appearing in \eqref{ss0} with $s(A) \le 0$ in the following real Jordan form:  
\begin{equation}\label{jordan}
 A = S \oplus  E_0 \oplus J(0, k_1) \oplus ... \oplus J(0 , k_p)  \oplus J(0,d_1, g_1) \oplus ... \oplus 
 J(0, d_q, g_q) \oplus E_1,  
\end{equation}
 $p,q \ge 1$. Note that some of the previous blocks could be not present (for instance, it could be possible that $A = S \oplus E_1).$
In the sequel we will  examine the various possibile   blocks in \eqref{jordan}.

\vv The  block $S$ is a $s \times s$ matrix, $s \ge 1,$  and corresponds to the stable part of $A$  (i.e., it corresponds to the eigenvalues of $A$  with negative real part). 

 The block  $E_0$ is the  null $k_0 \times k_0$ matrix, $k_0 \ge 1$.  The $2t \times 2t$ block  $E_1$, $t \ge 1,$ corresponds to all  possible simple complex  eigenvalues $i h_1,..., i h_t $ of $A$ (with $h_1, ...,h_t \in \R$): 
 $$  
E_1 = \left(\begin{array}{cccc} 
\left[\begin{array}{cc}
0 & h_1\\
- h_1 & 0
\end{array}\right] &    &    0\\
 & \ddots\\
 0 &  & \left[\begin{array}{cc}
0 & h_t\\
-h_t & 0
\end{array}\right] 
\end{array}\right). 
$$   
Moreover, $J(0,k_i)$ is the $k_i \times k_i$
 Jordan block, $k_i \ge 2,$ $i= 1,...,p$,
 $$   
 J(0,k_i) = \left(
           \begin{array}{ccccc}
             0  &  1  & 0 & \cdots  & 0 \\
              0  & 0 & 1 & \cdots  & 0 \\
                \cdots  &  \cdots &  \cdots & \cdots  & \cdots \\
                0  & 0 & 0 & 0  & 1
                \\
                0   & 0  & 0  & 0   & 0
            \end{array}   
          \right)    
 $$
 such that $J(0,k_i)^{k_i}$ is the null $k_i \times k_i$ matrix.  
Finally,  the $2g_j \times 2g_j $ Jordan block  $J(0,d_j, g_j)$ is 
   \begin{equation}  \label{jor2}  J(0,d_j, g_j) =\left(
\begin{array}{cccccccc} 
0 &   d_j     & 1 & 0  & 0 & 0 & 0 &0    \\
-d_j &  0     & 0  & 1  & \cdots &\cdots & \cdots  & \vdots   \\
\vdots  & \vdots  & \ddots & \ddots & \ddots &  \ddots & \ddots & \vdots          \\
0       & \cdots  & 0  & \ddots & \ddots &  \ddots & 1 & 0 \\
0       & \cdots  & 0  & \ddots & \ddots & \ddots & 0 & 1  \\
0       & \cdots  & 0  & \ddots &  \ddots &\ddots  & 0 &  d_j\\
0       & \cdots  & 0  & \ddots & \ddots & \ddots & -d_j & 0  \\
\end{array}
\right),  
\end{equation} 
where  $d_j \in \R$, $g_j \ge 2$, $j =1, ..., q$.  
If all blocks are present, according to   \eqref{jordan} we have   
$$
s + k_0 + k_1 + ... + k_p + 2[g_1 + ... + g_q ] + 2t = N.
$$
  Let us consider  \eqref{jordan} and fix  a real function $u(x_1, ..., x_N)$.

 In the following definition we suppose that  
 at least one   Jordan   block like $J(0, k_i)$ or  $ J(0,d_j, g_j)$ is  present
 in \eqref{jordan}.

\vv  We say that $u$ is {\it quasi-constant with respect to  Jordan  blocks
 like $J(0, k_i)$ or  $ J(0,d_j, g_j)$} if the following conditions hold:

\vv (1) If the block $J(0,k_i)$ is present in formula \eqref{jordan}
then  $u$ is constant in the variables
\begin{gather*}
 x_1^i, ..., x_{k_i -1}^i,
 \end{gather*}
 where $x_{ J(0,k_i)} = (x_1^i, ..., x_{k_i }^i)$
   (cf. \eqref{17bis};  i.e., if we  consider only the variables  $x_1^i, ..., x_{k_i }^i$, corresponding to the block  $J(0,k_i)$, $i=1,..., p$, we  say that $u$ may only depend on $x_{k_i }^i$.  

\vv  (2) If the block $J(0,d_j, g_j)$ is present in formula \eqref{jordan} then  $u$ is constant in the variables
\begin{gather*} 
 x_1^j, ..., x_{2 g_j -2}^j,   
\end{gather*}
 where $x_{J(0,d_j, g_j)} = (x_1^j, ..., x_{2 g_j }^j)$; 
i.e., if we  consider only the variables  $x_1^j, ..., x_{2 g_j }^j$, corresponding to    $J(0,d_j, g_j)$ we  say that $u$ may only depend on $x_{2 g_j -1 }^j$ and  $x_{2 g_j }^j$. 

\begin{lemma}\label{redaction} We assume the hypotheses    of Theorem \ref{main} about $L$ and $u$. Then $u$ is constant in the following cases.

\vv i) there is not the stable part $S$ in \eqref{jordan}. Moreover,
blocks like $J(0, k_i)$ or   $ J(0,d_j, g_j)$ are not present.  

ii)  there is not the stable part $S$ in \eqref{jordan}. Moreover, 
 $u$ is  quasi-constant with respect to  Jordan  blocks
 like $J(0, k_i)$ or  $ J(0,d_j, g_j)$.
 
iii)  there is  the stable part $S$ in \eqref{jordan}. Moreover,  
 $u$ is  quasi-constant with respect to  Jordan  blocks
 like $J(0, k_i)$ or  $ J(0,d_j, g_j)$.
\end{lemma}  
\begin{proof} i) In this case $L$ verifies the assumptions of Theorem \ref{ou1} and the assertion follows (note that in such case we do not need to impose any growth condition on the non-negative function $u$). 

\vv To treat (ii) and (iii)  we concentrate on   the most difficult case when  both blocks like $J(0, k_i)$ and   $ J(0,d_j, g_j)$ are  present in \eqref{jordan}
(otherwise, we can argue similarly).  

\vv \noindent ii) We start to study the term     $Ax \cdot Du$ (cf. \eqref{ss0}) .

\vv Let $1 \le i \le p.$
If $u$ is   constant in the variables   $x_1^i, ..., x_{k_i -1}^i$ corresponding to the block  $J(0,k_i)$
then, for any $x \in \R^N$ of the form
$$ 
(0, ..., 0,   x_{J(0,k_i)} ,0, ...,  0) =
(0, ..., 0,  x_1^i, ..., x_{k_i }^i ,0, ...,  0) 
$$ 
 (cf. \eqref{a34}) i.e.,    
 $x$ has  all  the coordinates 0 possibly apart the coordinates 
$x_1^i, ..., x_{k_i }^i $,  we have: 
\begin{gather}  \nonumber
A (0, ..., 0, x_1^i, ..., x_{k_i }^i, 0, ..., 0) \cdot Du(x) =  J(0 , k_i) (x_1^i, ..., x_{k_i }^i) \cdot  D_{(x_1^i, ..., x_{k_i }^i)} u(x)
\\ \label{ouu}
= 0 \cdot  \partial_{x_{k_i }^i } u(x) =0,
\end{gather}  
 where $ D_{(x_1^i, ..., x_{k_i }^i)} u(x) \in \R^{k_i}$  denotes the  gradient  with respect to  the      $x_1^i, ..., x_{k_i}^i$ variables of $u$ at $x \in \R^N$.   
  
\vv  Let $1 \le j \le q.$ If    $u$ is   constant in the variables   $x_1^j, ..., x_{2g_i-2}^j$ corresponding to  $ J(0,d_j , g_j)$ 
 then 
 \begin{gather}  \nonumber 
A (0, ..., 0, x_1^j, ..., x_{2g_j}^j, 0, ..., 0) \cdot Du(x) =   J(0,d_j , g_j)  ( x_1^j, ..., x_{2g_j}^j)  \cdot D_g u(x) 
\\ \nonumber 
= 
  \big( d_j \, x_{2g_j}^j,  
 - d_j  \,  x_{2g_j-1}^j \big)
\cdot 
 \big( \partial_{x_{2g_j -1}^j} u(x),   
 \partial_{x_{2g_j }^j} u(x)
 \big)     
\\ \label{see} = d_j \, x_{2g_j}^j \,  \partial_{x_{2g_j -1}^j} u(x) 
- d_j  \,  x_{2g_j-1}^j \,  \partial_{x_{2g_j }^j} u(x), 
\end{gather}
where $D_g u(x) = D_{ (x_1^j, ..., x_{2g_j}^j) }u(x) \in \R^{2 g_j}$  denotes the  gradient with respect to  the      $ x_1^j, ..., x_{2g_j}^j$ variables of $u$ at $x \in \R^N$.

\vv Note that  according to \eqref{jordan} and \eqref{stt} we have, for any $x \in \R^N$,  
\begin{gather*} 
 Ax = \Big (E_0 \pi_{E_0} x ,   J(0, k_1) \pi_{J(0, k_1)}x,  ... ,J(0, k_p)\pi_{J(0 , k_p)}x,  
 \\  \nonumber  J(0,d_1, g_1) \pi_{J(0,d_1, g_1)} x,  ... , 
J(0, d_q, g_q) \pi_{J(0, d_q, g_q)},  E_1 \pi_{E_1}x \Big).   
\end{gather*}  
 By the assumptions $u$ depends only on $m$ variables, $m \le N$. Taking into account  \eqref{ouu} and \eqref{see}  we get, for any  $x \in \R^N,$  
\begin{gather} \nonumber
 Ax \cdot Du(x) =   
 E_0 \pi_{E_0} x \cdot D_{E_0}u(x) +    
  [d_1 \, x_{2g_1}^1 \,  \partial_{x_{2g_1 -1}^1} u(x)  
- d_1  \,  x_{2g_1-1}^1 \,  \partial_{x_{2g_1 }^1} u(x)] 
\\ \label{abr}  +  ... +  
  [d_q \, x_{2g_q}^q \,  \partial_{x_{2g_q -1}^q} u(x) 
- d_q  \,  x_{2g_q-1}^q \,  \partial_{x_{2g_q }^q} u(x)]  
+ E_1 \pi_{E_1}x\cdot D_{E_1}u(x),
\end{gather}  
 where $D_{E_0}$ denotes the gradient with respect to the $k_0$ variables corresponding to $E_0$ and  $D_{E_1}$ denotes the gradient with respect to the $2t$ variables corresponding to $E_1$. 
 
We can set   
 \begin{gather} \label{uuu2}
u(x_1, ..., x_N) =  v (x_{1},..., x_m)= v(x_C),
\end{gather}
with  $v:   \R^m \to \R_+$,  
$$
x_C= ( x_{1},..., x_m) = (x_{E_0} ,x_{2g_1 -1}^1 ,x_{2g_1 }^1, ...,  
x_{2g_q -1}^q, x_{2g_q }^q,  x_{E_1}). 
$$
 By \eqref{abr} we see that    
$$
 Ax \cdot Du(x)  = A_C  x_C \cdot Dv(x_C)
$$
 for a suitable 
 $m \times m$-matrix $A_C$ which  is diagonalizable over the complex field with all the eigenvalues on the imaginary axis. We obtain 
 \begin{gather*}
0=  \tilde L v ( x_C) = \frac{1}{2}\text{ tr}  (\tilde Q D^2 v( x_C) ) + 
   A_C x_C \cdot  D v( x_C), \;\;\;\;  x_C\in \R^m, 
\end{gather*}
for  a suitable positive definite   $m \times m $ matrix   $\tilde Q$.
 By applying Theorem \ref{ou1} we obtain that $v $ is constant and so $u$ is constant as well.

\hh (iii)  We start as in (ii).
 We denote  the variables corresponding to the stable part   $S$ of $A$ by 
\begin{gather*}
x_1, ..., x_s.
\end{gather*}
 By the assumptions $u$ depends only on $n$ variables, $n \le N$. Moreover 
 $n = s+ m $ where $m$ is considered in \eqref{uuu2}.
 So we write
 \begin{gather} \label{uuu} 
u(x_1, ..., x_N) =  w (x_1, ...., x_s, x_{s+1},..., x_n)= w(x_S, x_C),
\end{gather}
where   $w : \R^n \to \R_+$, $x_S= (x_1, ...., x_s)$ and $x_C= ( x_{s+1},..., x_n)$. 
 
 Arguing as before we obtain that there exists an $(n-s) \times (n-s) $-matrix $A_C$ such that, for any $x = (x_S, x_C)\in \R^n$,
 \begin{gather*}
0=  \tilde L w (x) = \frac{1}{2}\text{ tr}  (\tilde Q D^2 w(x) ) +  Sx_S \cdot  D_S w(x) +
  A_C x_C \cdot  D_C w(x),   
\end{gather*}
for  a suitable positive definite $n \times n $ matrix   $\tilde Q$. Moreover, $D_S$ denotes the gradient with respect to  the   first $s$ variables and $D_C$  the gradient with respect to   the  $x_{s+1}, ..., x_n$ variables.

 In the rest of the proof it is convenient to replace $\tilde Q$ by
\begin{equation*}\label{dde} 
\delta \tilde Q,\;\;\;\; \text{for some } \;\; \delta > 0 \;\; \text{small enough to be chosen later.}   
\end{equation*}
(cf. Remark \ref{qq1}).
 We also recall that  the matrix $A_C$ is   diagonalizable  over the complex field
with all the eigenvalues on the imaginary axis. 
  
\vv { \sl To finish the proof we prove that  $w$ does not depend on the $x_S$-variable.
 Indeed once  this is proved we can apply Theorem \ref{ou1} and obtain that $w$ is constant.} 
 
 To obtain such  assertion 
 it will be important the exponential growth condition \eqref{e1}.

  By the  previous notation,   $x =  (x_S, x_C) \in \R^n,$
since in particular  $w$ verifies the exponential growth condition \eqref{e1} we have:
 $\tilde P_t w =w$, $t \ge 0,$ i.e., 
\begin{gather} \label{sss} 
  \int_{\R^{n}}  w (e^{t S}x_S \, + y_S, e^{t A_C}x_C \, + y_C )  
 \,  N(0 , \tilde Q_t) dy_{S} dy_C = w(x_S, x_C),
\end{gather}
 $t \ge 0,$  where $N(0, \tilde Q_t)$ is the Gaussian measure with mean 0 and covariance matrix 
 \begin{gather*} \label{madai}
 \tilde Q_t =     \delta \int_0^t e^{s \tilde A} \tilde Q e^{s \tilde A^*} ds.
 \end{gather*} 
  Here we are considering   the $n\times n$-matrix  $\tilde A = S \oplus A_C $ so that   
 \begin{gather*} 
e^{t \tilde A} = e^{t S} \oplus e^{t A_C}.  
\end{gather*} 
   By Corollary  \ref{ci2}    $w$ is a convex function on $\R^n$. 
Applying a well-known result on convex functions (cf. Section 6.3 in \cite{EG}) we obtain in particular that, for any $x \in \R^n$, $|x| > 1$, 
\begin{gather*}  
\sup_{|y| \le |x| } |Dw(y) | \le   \frac{c(n)}{2|x|} 
\, \fint_{\mathrlap{B(0, |2 x|)} }\,\;\;\;\;\;\;\; w(y) dy. 
\end{gather*} 
It follows that possibly replacing $c_{0}$ in \eqref{e1} by another constant $c >0$  
we have
\begin{equation}\label{gro1}
|Dw(x) | \le c\, e^{c\, |x|},\;\;\; x \in \R^n.
\end{equation}
   Let us fix $h_S \in \R^{n-s}$ and $x = (x_S, x_C) \in \R^n$.   Differentiating both sides of \eqref{sss} along the direction $h_S$ we find 
\begin{equation}\label{diff} 
\int_{\R^{n}}\! \!\!  D_S w (e^{t S}x_S \, + y_S, e^{t A_C}x_C \, + y_C ) \cdot e^{tS} h_S  
 \,  N(0 , \tilde Q_t) dy_{S} dy_C =  D_S w(x_S, x_C) \cdot h_S.
\end{equation}
where as before $D_S$ denotes the gradient with respect to the first $s$ variables. Recall that since the matrix $S$ is stable  there exist $C >0$ and $\omega >0$ such that 
\begin{equation}\label{st2}
|e^{tS} h_S | \le C e^{- \omega t} |h_S|,\;\; t \ge 0.
\end{equation}
By \eqref{gro1} we infer  
\begin{gather*}
|D_S w (e^{t S}x_S \, + y_S, e^{t A_C}x_C \, + y_C )| \le c e^{c |e^{t S}x_S|}  e^{c 
|e^{t A_C}x_C|} \, e^{c |y_S|} e^{c |y_C|}. 
\end{gather*}
Note that  
$|e^{t A_C}x_C| =  |x_C|, \;\; t \ge 0$.  It follows that there exists a positive function   $\lambda(x)$
(independent of $t \ge 0$)    such that 
\begin{equation}\label{ss1}
|D_S w (e^{t S}x_S \, + y_S, e^{t A_C}x_C \, + y_C )| \le \lambda(x)  \,   e^{2c \, |(y_S, y_C)|},\;\;\; t \ge 0.    
\end{equation} 
Setting $y= (y_S, y_C)$ it is not difficult to prove that there exists $c_1 >0$
 (independent of $t$) such that 
\begin{gather} \label{d41} 
  \int_{\R^{n}}  e^{2c\,  |y|} N(0 , \tilde Q_t) dy  \le  
  c_1 e^{c_1 \, c^{2}\,  \delta  \, t }.
\end{gather}  
To this purpose  we first remark that  
\begin{equation}\label{stt5}
\| \tilde Q_t\| \le \delta \int_0^t \| e^{s \tilde A} \| \,  \| Q\|  \, \| e^{s \tilde A^*}  \| ds \le  C_0 \delta t, \;\; t \ge 0,
\end{equation} 
for some constant $C_0 >0$ independent of $t$.
Then recall  that  if $R$ is a $n \times n$ symmetric and non-negative definite matrix,  using also the Fubini theorem, we obtain 
\begin{gather*}
\int_{\R^{n}} e^{r |y|} N(0, R)dy = \int_{\R^{n}} e^{r |R^{1/2} y|} N(0, I)dy
\\  
\le \int_{\R^{n}} e^{r \|R^{1/2}\| \, (|y_1| + ... + |y_n|) } 
N(0, I)dy = \Big (\frac{2}{\sqrt{2 \pi}} \int_0^{\infty}  e^{r \|R^{1/2}\| \, y } \, e^{- y^2/2} dy \Big)^n
\\ 
\le 2^{n} e^{\frac{n}{2} r^{2} \|R \| },\;\; r \ge 0.   
\end{gather*}
Combining the last computation and  \eqref{stt5} we obtain \eqref{d41}.  Using \eqref{diff}, \eqref{ss1} and \eqref{d41} and  we infer 
\begin{gather*}
 |  D_S w(x_S, x_C) \cdot h_S|  
 \\ 
=\Big | \int_{\R^{n}} D_S w (e^{t S}x_S \, + y_S, e^{t A_C}x_C \, + y_C ) \cdot e^{tS} h_S  
 \,  N(0 , \tilde Q_t) dy_{S} dy_C \Big| \\ 
 \le \lambda(x) \, |e^{tS} h_S   | \, \int_{\R^{n}}  e^{2c\, |y|} N(0 , \tilde Q_t) dy  \,
  \le   
  c_1 C  \lambda(x) \, e^{c_1 c^{2}\,  \delta  \, t  }  \, e^{-\omega t}|h_S|,\;\; t \ge 0.   
\end{gather*}
 Since $c_1$ and $c$ are independent of $t$ and $\omega$,
choosing $\delta >0$ small enough and passing to the limit as $t \to \infty$,
we get  
 \begin{gather*}
D_S w(x_S, x_C) \cdot h_S =0.
\end{gather*}
 It follows that    $w$ does not depend on the $x_S$-variable.
We have $w(x_S, x_C) = g(x_C)$ for a regular function $g: \R^{n-s} \to \R_+$.
 Moreover,   
\begin{gather*}
\tilde L w(x_S, x_C) =  \frac{\delta}{2}\text{ tr}  (\tilde Q_0 D^2_C \, g(x_C)  ) +  A_C x_C \cdot  D_C g (x_C) =0, \;\; x_C \in \R^{n-s}, 
\end{gather*}  
 where $\tilde Q_0$ is a positive definite $(n-s) \times (n-s)$-matrix.

 Applying  Theorem \ref{ou1} to the OU operator $ \frac{\delta}{2}\text{ tr}  (\tilde Q_0 D^2_C    ) +  A_C x_C \cdot  D_C   $ we obtain that $g$ is constant and this finishes the proof.   
  \end{proof}


 \def\ciau{
 \begin{remark} \label{aiuto} {\em  To explain the proof of Lemma  \ref{n} 
 we consider the following example
 Let $N=2$ and
\begin{gather*} 
Q = \left(
           \begin{array}{cc}
             a  &  0  \\
              0  & 1
            \end{array} 
          \right), \; a \ge 0, \;
A = \left(
           \begin{array}{cc}
             0  &  1  \\
              0  & 0
            \end{array} 
          \right), \;\;\;\;  e^{tA} x= \left(
           \begin{array}{c}
             x_1 + t x_2    \\
               x_2  
            \end{array} 
          \right)
\end{gather*} 
 (note that in this example  $Q$ could be degenerate).
 We fix any $x_0 \in \R^2$. We have to prove that $\partial_{x_1}u(x_0)=0$. Suppose by contradiction that 
 \begin{gather*}
k = \partial_{x_1}u(x_0) \not =0.
\end{gather*}
 Then we consider  $x= (0, k)$. We find by \eqref{serve} for any $t \ge 0$
 \begin{gather*}
 u(0, k) \ge  u(x_0) -  Du(x_0)\cdot x_0  +   Du(x_0)\cdot e^{tA} ((0, k))
 \\
 = u(x_0) -  Du(x_0)\cdot x_0  + Du(x_0)\cdot (tk, k) 
 \\
 =  u(x_0) -  Du(x_0)\cdot x_0  + t (\partial_{x_1}u(x_0))^2
  + \partial_{x_2}u(x_0) \partial_{x_1}u(x_0).
\end{gather*}
 Letting $t \to \infty$ we get a contradiction. It follows that 
 $\partial_{x_1}u(x_0) \not =0$. We have 
 $
u(x_1, x_2) = u(x_2)$   on $ \R^2. $
Since   $0= au_{x_1 x_1} + u_{x_2 x_2} + x_2 u_{x_1} = u_{x_2 x_2}$. This  implies that $u=cost$. }
\end{remark}
} 
  
  \subsection { Proof  of Theorem \ref{main} } 
We concentrate on   the most difficult case when  both blocks like $J(0, k_i)$ and   $ J(0,d_j, g_j)$ are  present in \eqref{jordan}.   
  By  Lemma \ref{redaction} it is enough to show that  
  $u$ is   quasi-constant  with respect to the  Jordan  blocks 
 $$
 J(0, k_1),  ... ,J(0 , k_p),  J(0,d_1, g_1),  ... ,J(0, d_q, g_q).
 $$

\hh {\it I Step. }
We fix $i =1, .., p$, and consider the block $J(0, k_i)$ (see \eqref{jordan}).  Let  $x_1^i, ..., x_{k_i }^i $ be the variables corresponding to  $J(0, k_i)$ according to \eqref{17bis}. 
Let $x_0 \in \R^N$. We prove that $\partial_{x^i_k} u (x_0) =0$ when $k =1, ..., k_i-1$. 

\vv To this purpose we first consider $k=1$.
 We    argue by contradiction and suppose that 
\begin{equation}\label{ma}
\partial_{x^i_1} u (x_0) \not = 0,  
\end{equation}
 for some $x_0 \in \R^N$. In order to apply  Lemma \ref{n}, we  first  choose $x$ having 0 in all  the coordinates apart the coordinates 
$x_1^i, ..., x_{k_i }^i $, i.e., 
we have  $x = (0, ..., 0,  x_1^i, ..., x_{k_i }^i ,0, ...,  0)$.  We find, setting $M_{x_0} = u(x_0) -  Du(x_0)\cdot x_0$ 
\begin{gather*}
u(0, ..., 0,  x_1^i, ..., x_{k_i }^i ,0, ...,  0) \ge M_{x_0}  +   Du(x_0)\cdot e^{tA} (0, ..., 0,  x_1^i, ..., x_{k_i }^i ,0, ...,  0), 
\\
=  M_{x_0}  +   D_{(x_1^i, ..., x_{k_i }^i)}u(x_0)\cdot e^{t\,  J(0, k_i)} ( x_1^i, ..., x_{k_i }^i), \;\; t \ge 0,
\end{gather*}
where $D_{(x_1^i, ..., x_{k_i }^i)}u(x_0)$ denotes the gradient with respect to  the variables
 $x_1^i, ..., x_{k_i }^i$. Recall that 
$$
  e^{t\,  J(0, k_i)} =  
  \left(\begin{array}{cccccc}
      1&t&\frac{t^2}{2!}&\dots&\dots&\frac{t^{k_{i}-1}}{(k_{i}-1)!}\\[4pt]
      0& 1&t&\dots&\dots&\frac{t^{k_{i}-2}}{(k_{i}-2)!}\\[4pt]
      0& 0& 1&\dots&\dots&\frac{t^{k_{i}-3}}{(k_{i}-3)!}\\[4pt]
      &\ddots & &\vdots&\vdots&\vdots\\[4pt]
      0&\dots&\dots&0&1&t\\[4pt]
      0&\dots&\dots&\dots&0&1\\[4pt]
  \end{array} \right).
$$
Choosing further $x_1^1=0, ...,  x_{k_i-1}^i=0$, $x_{k_i}^i = \partial_{x_{1}^i}u(x_0)$  
we find 
\begin{gather*} 
u(0, ...,0, \partial_{x_{1}^i}u(x_0) ,0,  ..., 0) \ge M_{x_0}   +   \partial_{x_{1}^i}u(x_0)
\, \big[0 + t 0 +... + \frac{t^{k_i -1}}{ (k_i-1)!} x_{{k_i}}^i \big] 
\\ + p(t,x_0)
=  M_{x_0}   +   (\partial_{x_{1}^i}u(x_0))^2 \,  
 \frac{t^{k_i -1 }}{(k_i -1)!} +  p(t,x_0),   \;\; t \ge 0,
\end{gather*} 
where $p(t,x_0)$ is a polynomial in the  $t$-variable which has degree less than 
$k_i-1$. Letting $t \to \infty$ we find a contradiction since  $(\partial_{x_{1}^i}u(x_0))^2 \, 
 \frac{t^{k_i -1}}{(k_i -1)!} + p(t, x_0)$ tends to $\infty$. Hence \eqref{ma} cannot hold and we have proved that $u$ does not depend on the $x_1^i$-variable.
 
Similarly, we prove that    $u$ does not depend on the $x_2^i$-variable as well.
 
Proceeding in finite steps, once  we have proved that $u$   does not depend on the variables $x_1^1, ..., x_{k-1}^i$,   $k=1, ..., k_i-1$, we can show that for any $x_0 \in \R^N$ we have  $\partial_{x^i_k} u (x_0) =0$. To this purpose
 we    argue  by contradiction and suppose that 
\begin{equation}\label{ma1} 
\partial_{x^i_k} u (x_0) \not = 0,
\end{equation}
 for some $x_0 \in \R^N$.
In order to apply  Lemma \ref{n}, we  choose $x$ having 0 in all  the coordinates apart the coordinates 
$x_1^i, ..., x_{k_i }^i $. Moreover,  choosing further  $x_1^1=0, ...,  x_{k_i-1}^i=0$, $x_{k_i}^i = \partial_{x_{k}^i}u(x_0)$  
we find (using that $u$   does not depend on the variables $x_1^1, ..., x_{k-1}^i$)
\begin{gather*} 
u(0, ...,0, \partial_{x_{k}^i}u(x_0) ,0,  ..., 0) \\ \ge M_{x_0}   +   \partial_{x_{k}^i}u(x_0)
\, \big[x_k^i + t x_{x_{k+1}}^i +... + \frac{t^{k_i -k}}{(k_i -k)!} x_{{k_i}}^i \big] 
\\
+ q_k(t,x_0) =  M_{x_0}   +   (\partial_{x_{k}^i}u(x_0))^2 \, 
 \frac{t^{k_i -k}}{(k_i -k)!} + q(t,x_0),  \;\; t \ge 0,
\end{gather*}
where $q(t,x_0)$ is a polynomial in the  $t$-variable which has degree less than 
$k_i-k$.

Letting $t \to \infty$ we find a contradiction since  $(\partial_{x_{k}^i}u(x_0))^2 \, 
 \frac{t^{k_i -k}}{(k_i -k)!} $ $+q(t,x_0)$ tends to $\infty$. Hence \eqref{ma1} cannot hold and we have proved the assertion.

\hh {\it II Step. } We fix $j =1, .., q$  and consider the block $J(0, d_j, g_j)$ (see \eqref{jordan}).  Let  $x_1^j, ..., x_{2 g_j }^j $ be the variables corresponding to  $J(0, d_j, g_j)$.  
Let $x_0 \in \R^N$. We prove that $\partial_{x^j_k} u (x_0) =0$ when $k =1, ..., 2g_j -2$. 
 
\vv We first consider $k=1$. 
 We    argue by contradiction and suppose that 
\begin{equation}\label{ma2}
\partial_{x^j_1} u (x_0) \not = 0,
\end{equation}
 for some $x_0 \in \R^N$. As in I Step  in order to apply  Lemma \ref{n}, 
 we  choose $x$ having 0 in all  the coordinates apart the coordinates 
$x_1^j, ..., x_{2g_j  }^j $. Moreover, we  choose further  
$$
x_1^j=0, ...,  x_{2g_j-1}^j=0, x_{2g_j}^j = \partial_{x_{1}^j}u(x_0)
$$   
  and  set $M_{x_0} = u(x_0) -  Du(x_0)\cdot x_0$.
By considering $t = T_n$, with 
\begin{equation}\label{drf}  
d_j \cdot  T_n = \frac{\pi}{2} + 2n \pi,\;\;\; n \ge 0
\end{equation}
we find 
\begin{gather*}
u(0, ...,0, \partial_{x_{1}^j}u(x_0) ,0,  ..., 0) \ge M_{x_0}     
\\
+ \partial_{x_{1}^j}u(x_0)
\,   \big[x_1^j \cos(d_j  T_n) + x_2^j \sin(d_j  T_n) + x_3^j T_n  \cos(d_j  T_n) + x_4^j T_n\sin(
d_j  T_n)
\\ + ...+ x_{2 g_j -1}^j \frac{ T_n^{g_j -1}}{(g_j -1)!}  \cos(d_j  T_n) + x_{2 g_j }^j 
\frac{ T_n^{g_j -1}}{(g_j -1)!} \sin(d_j  T_n) \big] 
 \\
+ p(T_n,x_0) =  M_{x_0}   +   (\partial_{x_{1}^j}u(x_0))^2 \,  \frac{ T_n^{g_j -1}}{(g_j -1)!}
 + p(T_n,x_0),   \;\; n \ge 0,
\end{gather*} 
where $p(t,x_0)$ is a polynomial in the  $t$-variable which has degree less than 
$g_j-1$. Letting $n \to \infty$ we find a contradiction since  $(\partial_{x_{k}^j}u(x_0))^2 \, 
\frac{ T_n^{g_j -1}}{(g_j -1)!} + p(T_n,x_0)$ tends to $\infty$.  Hence \eqref{ma2} cannot hold and we have proved  
that $u$ does not depend on the $x_1^j$-variable.

\vv   
Similarly, one can prove that    $u$ does not depend on the $x_2^j$-variable as well. We only note that in this case we choose   $x$ having 0 in all  the coordinates apart the coordinates 
$x_1^j, ..., x_{2g_j  }^j $. Moreover, $x_1^j=0, ...,  x_{2g_j-1}^j=0$,  $x_{2g_j}^j = \partial_{x_{2}^j}u(x_0)$ and define  
  $T_n$  such that $d_j  \cdot  T_n =  2n \pi,$ $  n \ge 0$. We  have
\begin{gather*}
u(0, ...,0, \partial_{x_{1}^j}u(x_0) ,0,  ..., 0) \ge M_{x_0} 
\\ 
+ \partial_{x_{2}^j}u(x_0) 
\,   \big[- x_1^j \sin(d_j  T_n) + x_2^j \cos(d_j  T_n)  -x_3^j T_n  \sin(d_j  T_n) + x_4^j T_n\cos(d_j  T_n)  
\\ + ... - x_{2 g_j -1}^j \frac{ T_n^{g_j -1}}{(g_j -1)!}  \sin(d_j  T_n) + x_{2 g_j }^j 
\frac{ T_n^{g_j -1}}{(g_j -1)!} \cos(d_j  T_n) \big] 
 \\  
+ q(T_n,x_0),   \;\; n \ge 0, 
\end{gather*}  
where $q(t,x_0)$ is a polynomial in the  $t$-variable which has degree less than 
$g_j-1$. 
   
\vv 
 Proceeding in finite steps, once  we have proved that $u$   does not depend on the variables $x_1^1,$ ...,$x_{k-1}^j$, $k=1, ..., 2g_j-2,$   we  show that for any $x_0 \in \R^N$ we have  $\partial_{x^j_k} u (x_0) =0$.   To this purpose we suppose that $k$ is even (we can proceed  similarly if $k$ is odd).
  We argue by contradiction and suppose that 
\begin{equation}\label{ma13}
\partial_{x^j_k} u (x_0) \not = 0,
\end{equation}
 for some $x_0 \in \R^N$. 
We  choose $x$ having 0 in all  the coordinates apart the coordinates 
$x_1^j, ..., x_{2 g_j  }^j $. Moreover,   we set   $x_1^1=0, ...,  x_{2 g_j-1}^j=0$, $x_{2g_j}^j = \partial_{x_{k}^j}u(x_0)$.  
 We find (using that $u$   does not depend on the variables $x_1^1, ..., x_{k-1}^j$) with $T_n$  as in \eqref{drf}: 
\begin{gather*}
u(0, ...,0, \partial_{x_{k}^j}u(x_0) ,0,  ..., 0) \ge M_{x_0}    
\\
+ \partial_{x_{k}^j}u(x_0)
\,   \big[x_k^j \cos(d_j T_n) + x_{k+1}^j \sin(d_j  T_n) + x_{k+2}^j T_n  \cos(d_j  T_n) + x_{k+3}^j T_n\sin(d_j  T_n)
\\ + ...+ x_{2 g_j -1}^j \frac{ T_n^{g_j - \frac{k+1}{2} }}{(g_j - \frac{k+1}{2} )!}  \cos(d_j  T_n) + x_{2 g_j }^j 
 \frac{ T_n^{g_j - \frac{k+1}{2} }}{(g_j - \frac{k+1}{2} )!}  \sin(d_j  T_n) \big] 
 \\
+ h(T_n,x_0) =  M_{x_0}   +   (\partial_{x_{k}^j}u(x_0))^2 \,  \frac{ T_n^{g_j - \frac{k+1}{2} }}{(g_j - \frac{k+1}{2} )!}  
 + h(T_n,x_0),     \;\; n \ge 0,   
\end{gather*}
where $h(t,x_0)$ is a polynomial in the  $t$-variable which has degree less than 
$ g_j - \frac{k+1}{2} .$ Letting $n \to \infty$ we find a contradiction since  $(\partial_{x_{k}^j}u(x_0))^2 \,\frac{ T_n^{g_j - \frac{k+1}{2} }}{(g_j - \frac{k+1}{2} )!}$ $ + h(T_n,x_0  )
  $ tends to $\infty$. Hence \eqref{ma13} cannot hold and we have proved the assertion. The proof is complete. \qed 
        

\section{Some open problems}  \label{open}

We list some open problems related to Liouville type theorems for OU  operators.

\hh (1) In general it is not known if under the Kalman condition \eqref{kal}
and the condition $s(A) \le 0$ all non-negative smooth solutions $v$ to
$
Lv =0
$ on $\R^N$ are  constant.    
 This problem is also open in the non-degenerate case  when we assume in addition that $Q$ is positive definite.

\hh (2)
 The papers \cite{Wang11} and \cite{Schilling} treat 
    non-degenerate purely non-local OU operators $L$   
   $$ 
 Lf(x)=  \int_{\R^N} \big(  f (x+y) -  f (x)
   - 1_{ \{  |y| \le 1 \} } \, \langle y , D f (x) \rangle \big)
   \, \nu (dy) + Ax \cdot D f(x), 
 $$
 $x \in \R^N$, $f: \R^N  \to \R$  bounded and  smooth, where 
$\nu$ is a 
 L\`evy measure.
 The hypotheses  of \cite{Schilling} on $\nu$ improve the ones in \cite{Wang11}.
   Theorem 1.1 in \cite{Schilling}  
 shows that under suitable hypotheses on $\nu$ and assuming 
 $ \sup_{t \ge 0}\| e^{tA} \| < \infty$   all  bounded smooth harmonic functions  for $L$ are constant. It is not known if such result holds more generally under the assumption that $s(A) \le 0$  (for instance, a matrix like $A$ in Example \ref{exam} is not covered in \cite{Wang11} and \cite{Schilling}). 
  
\hh (3)  In   \cite{KLP}   
the  result below has been proved  
using probabilistic methods based on the known characterization of recurrence for    OU stochastic processes.
It seems that a purely analytic proof  of such result  is not known.

 \begin{theorem}[Theorem 6.1 in \cite{KLP}]\label{ou2} Let us consider hypoelliptic OU operator $L$.  Let $v: \R^N \to \R$ be a non-negative $C^2$-function such that 
$ L v \le 0$ on $\R^N$. Then $v$ is constant  if the following condition holds:

\vskip 1mm \noindent   The  real Jordan  representation of $B$ is
$ \begin{pmatrix}
 B_0 & 0\\ 
 0 & B_1 
\end{pmatrix}$
where $B_0$ is stable and  $B_1$ is at most of dimension 2 and of the form
 $B_1 = [0]$ or 
 $B_1 =  \begin{pmatrix}
 0  & -\alpha \\
 \alpha & 0
\end{pmatrix}
$
 for some $\alpha \in \R$ (in this case we need $N \ge 2$).
   \end{theorem}

 \appendix

 \section{Proof of Theorem    \ref{ci}}   

The proof of Theorem     \ref{ci}  is based on the following lemma which is a special case of an infinite dimensional result proved in  Section 5 in  \cite{PZproc}. We include the proof for the sake of completeness. 

\begin{lemma} \label{kw} Let $(P_t)$ be the OU semigroup. Assume \eqref{kal} and $s(A) \le 0$. Then 
   for any non-negative Borel function
 $f : \R^N \to \R$, there results: 
 \begin{equation} 
 \label{harmonic}
 P_t f (x+a ) + P_t f(x-a ) \ge 2 C_t(a) \, P_t f(x), \;\; x,\, a \in \R^N,
 \end{equation}
  where $ C_t (a) = \exp [-\frac{1}{2}
|Q_t^{-1/2} e^{tA } a|^2 ],\;  t>0 $
  (note that both sides in \eqref{harmonic} can be $+
  \infty$).   
\end{lemma} 
\begin{proof}   We fix $x, a \in \R^N$ and set $N(0,Q_t) = N_{0 ,Q_t}.$
By a direct  computation we have 
\begin{eqnarray*} 
  P_t  f (x+a)=   \int_{\R^N}  f (e^{tA}x  + y )  \frac{ dN
_{e^{tA} a ,Q_t }  } { dN_{0 ,Q_t} }(y) N_{0 ,Q_t}(dy)
\\
=  \int_{\R^N}  f (e^{tA}x  + y   )   \exp
[-\frac{1}{2} |Q_t^{-1/2} e^{tA } a|^2 +  \lan Q_t^{-1/2} e^{tA} a
, Q_t^{-1/2}y \ran ] N_{0 ,Q_t}(dy). 
\end{eqnarray*}
Note that the previous identity also holds in infinite dimensions by the   Cameron-Martin formula (see, for instance, Chapter 1 in \cite{DP02}). 
 It follows that
\begin{eqnarray*}     
\frac {1}{2} ( P_t f (x+a) +  P_t f  (x-a ))
\\
 =  e^{ -\frac{1}{2} |Q_t^{-1/2} e^{tA } a|^2 }\,
   \int_{\R^N}  f (e^{tA}x  + y  )
  \frac{1}{2}
  \Big ( e^{ \lan Q_t^{-1/2} e^{tA} a , Q_t^{-1/2}y \ran }
\\
  + \, e^{ - \lan Q_t^{-1/2} e^{tA} a , Q_t^{-1/2}y \ran } \Big) N_{0 ,Q_t}(dy);
\\
\ge  \exp [-\frac{1}{2} |Q_t^{-1/2} e^{tA } a|^2 ] \,  \int_{\R^N}  f (e^{tA}x  + y  )    N_{0 ,Q_t}(dy)
\\
 = \, C_t (a)\,   P_t f 
(x),\;\;\; \mbox {\rm where } \; \, C_t (a) = \exp [-\frac{1}{2}
|Q_t^{-1/2} e^{tA } a|^2 ]. 
 \end{eqnarray*}
 \end{proof}
{\bf Proof of Theorem \ref{ci}}. Let $u $ be a positive harmonic function for $(P_t)$.   By the previous lemma, we
have, for any $x, a \in \R^N$,
\begin{eqnarray*} 
\frac {1}{2} (u(x+a)  + u (x-a) ) = \frac {1}{2} (P_t u(x+a) + P_t
u (x-a) )
\\
\ge  \exp [-\frac{1}{2} |Q_t^{-1/2} e^{tA } a|^2 ] P_t u (x) =
\exp [-\frac{1}{2} |Q_t^{-1/2} e^{tA } a|^2 ]  u (x).
 \end{eqnarray*}
Passing  to the limit as $t \to \infty$, we infer by
Theorem \ref{nullo} 
 $$
 \frac {1}{2} (u(x+a)  + u (x-a) ) \ge   u
(x),\;\; x,\, a \in {\R^N}.
 $$ 
 By a well-known result due to  W. Sierpi\'ski
this condition together with the measurability of $u$ imply
the convexity of $u$. \qed

\ \hh {\bf Acknowledgments.} The author wishes to
thank  J. Zabczyk  for useful
discussions.



\begin{thebibliography}{pippo1} 

\bibitem{BF} M. Bertoldi, S.  Fornaro,  Gradient estimates in
parabolic problems with unbounded coefficients, 
 Studia Math. 165, 2004,  221–254. 

\bibitem{Baldi} P.  Baldi,  Stochastic calculus. An introduction through theory and exercises. Universitext. Springer, Cham, 2017. 
 
\bibitem{COR} M. Cranston, S. Orey, U.  R\"osler,  The Martin boundary of
two dimensional Ornstein-Uhlenbeck processes,  Probability,
statistics and analyis, London Math. Soc. Lect. Note Ser. 79,
  Cambridge Univ. Press, 1983, 63-78. 


\bibitem {DP02}G. Da Prato and J. Zabczyk, Second order partial
 differential equations in Hilbert spaces. London Mathematical Society Note
 Series, 293, Cambridge University Press, Cambridge, 2002.

 \bibitem {DP14}
G. Da Prato, J. Zabczyk, Stochastic equations in infinite dimensions. 
Second edition. Encyclopedia of Mathematics and its Applications, 152. Cambridge University Press, Cambridge, 2014.
 
 
 

\bibitem{EG} Evans L. C., Gariepy R. F.,  Measure Theory And Fine Properties of Functions, Routledge 1992. 

\bibitem{IW}
N. Ikeda, S. Watanabe, S., Stochastic Differential Equations
 and Diffusion Processes. North Holland-Kodansha, II edition (1989).

 
\bibitem {KL0} A. Kogoj, E. Lanconelli,
  An invariant Harnack inequality for a class of hypoelliptic ultraparabolic
  equations,   1,  2004, no. 1, 51-80.

\bibitem{KL1} A. E. Kogoj and E. Lanconelli, One-side Liouville theorems for a class of hypoelliptic ultraparabolic equations, in Geometric analysis of PDE and several complex variables,  
vol. 368 of Contemp. Math., Amer. Math. Soc., Providence, RI, 2005, pp. 305-312.
. 


\bibitem{KL2}  
A.~E.~Kogoj and E.~Lanconelli.
\newblock Liouville theorems for a class of linear second-order operators with
  non-negative characteristic form.
\newblock {\em Bound. Value Probl.}, Art. ID 48232, 16 pages, 2007.



\bibitem{KLP}  A. E.  Kogoj, E. Lanconelli, E.  Priola,  Harnack inequality and Liouville-type theorems for Ornstein-Uhlenbeck and Kolmogorov operators. Math. Eng. 2 (2020), no. 4, 680-697. 





\bibitem{kup_1972}
L.~P. Kupcov.
\newblock The fundamental solutions of a certain class of elliptic-parabolic
  second order equations.
\newblock {\em Differ. Uravn.}, 8:1649--1660, 1716, 1972.

\bibitem{lanconelli_polidoro_1994}
E.~Lanconelli and S.~Polidoro.
\newblock On a class of hypoelliptic evolution operators.
\newblock {\em Rend. Sem. Mat. Univ. Politec. Torino}, 52(1):29--63, 1994.
\newblock Partial differential equations, II (Turin, 1993). 

\bibitem{Pinskii} R. Pinsky,  Positive Harmonic Functions and Diffusion, Cambridge Univ. Press, 1995.  




\bibitem{PZnull} E. Priola, J. Zabczyk, Null controllability with vanishing energy,
{ SIAM J. Control Optim. ,}   42,  2003,
1013-1032.

 

\bibitem{PZjfa} E. Priola, J. Zabczyk, Liouville theorems for nonlocal operators,
{ . J. Funct. Anal.,} 216,  2004,
 455-490.


\bibitem{PZproc} E. Priola, J. Zabczyk,   Harmonic functions for generalized Mehler semigroups. SPDEs  and applications-VII, pages 243-256, Lect. Notes Pure Appl. Math., 245, Chapman  Hall/CRC,  2006, available at 
 \url{https://iris.unito.it/retrieve/handle/2318/62663/755313/PriolaZabczyk_Mehlerquad.pdf}
 
 
\bibitem{PW}   E.  Priola, F. Y. Wang,  Gradient estimates for diffusion semigroups with singular coefficients. . J. Funct. Anal. 236, 244-264, 2006. 

 
\bibitem{Schilling}  R. L. Schilling, P.  Sztonyk, J.  Wang,  On the coupling property and the Liouville theorem for Ornstein–Uhlenbeck processes,  J. Evol. Equ. 12,  2012, no. 1, 119-140.


\bibitem{Wang11}  F. Y. Wang,  
Coupling for Ornstein-Uhlenbeck processes with jumps, 
Bernoulli  17,  2011, no. 4, 1136-1158. 

\bibitem{zabczyk} J. Zabczyk, Mathematical control theory-an introduction. Second edition.  Birkh\"auser/Springer, Cham, 2020.   

\end{thebibliography}
\end{document}